\documentclass[a4,11pt]{article}

\usepackage{paralist, amsmath, amsthm, amsfonts, amssymb, xy}
\usepackage[margin=1in]{geometry} 
\usepackage{hyperref}
\usepackage[T1]{fontenc}
\usepackage{calligra}

\input xy
\xyoption{all}

\newcommand{\NN}{\mathbb{N}}
\newcommand{\ZZ}{\mathbb{Z}}
\newcommand{\SM}{\mathbf{Sm}_K}

\newcommand{\srarrow}{\twoheadrightarrow}
\newcommand{\irarrow}{\hookrightarrow}

\newcommand{\flag}{\mathcal{F}\ell\,}
\newcommand{\Proj}{\mathbb{P}}
\newcommand{\Aff}{\mathbb{A}}
\newcommand{\spol}{\mathfrak{S}}
\newcommand{\gpol}{\mathfrak{G}}
\newcommand{\bpol}{\mathfrak{H}}

\newcommand{\trecd}{\cdot\cdot\cdot}
\newcommand{\tred}{\ldots}
\newcommand{\unddot}{_\textbf{\textbullet}}

\newcommand{\spec}{{\rm Spec\,}}

\newtheorem{theorem}{Theorem}[section]
\newtheorem{lemma}[theorem]{Lemma}
\newtheorem{proposition}[theorem]{Proposition}
\newtheorem{corollary}[theorem]{Corollary}

\theoremstyle{definition} \newtheorem{remark}[theorem]{Remark}
\theoremstyle{definition} \newtheorem{example}[theorem]{Example}
\theoremstyle{definition} 
\theoremstyle{definition} 
\date{}

\title{\textbf{On the $K$-theoretic fundamental class of Deligne--Lusztig varieties}}
\author{THOMAS  HUDSON, DENNIS PETERS}

\begin{document}
\maketitle

\begin{abstract}
In this paper we express the class of the structure sheaves of the closures of Deligne--Lusztig varieties as explicit double Grothendieck polynomials in the first Chern classes of appropriate line bundles on the ambient flag variety. This is achieved by viewing such closures as degeneracy loci of morphisms of vector bundles.


\noindent \textit{2010 MCS:} Primary: 20G40, 19E08; Secondary: 14M15.
\end{abstract}

\section{Introduction}
The goal of this paper is to compute, through the use of universal polynomials, the fundamental classes of the closures of Deligne--Lusztig varieties in $K^0(Fl_n)$, the Grothendieck ring of vector bundles. These locally closed subvarieties of the flag varieties defined over fields of positive characteristic were introduced in \cite{RepresentationDeligne} by Deligne and Lusztig and play a fundamental role in the representation theory of finite groups of Lie type \cite{CharactersLusztig}. Recently, in \cite{HomologyKim}, Kim gave a formula which expresses their Chow ring fundamental class in terms of Schubert classes and, in the special case of flag varieties of type A, he was able to rewrite this expression using double Schubert polynomials.

 These universal polynomials in two sets of variables  indexed by the symmetric group $\{\mathfrak{S}_w(\mathbf{x},\mathbf{y})\}_{w\in S_\infty}$ were introduced by Lascoux and Sch\"utzenberger in \cite{LascouxSchutzenberger,ClassesLascoux}. Later, in \cite{FlagsFulton}, Fulton used them to describe the fundamental classes of the degeneracy loci of morphisms of vector bundles in the Chow ring $CH^*$ .      
   This result turned out to have analogues in $K^0$ (due to Buch \cite{GrothendieckBuch}) and, in characteristic 0, connective $K$-theory $CK^*$ \cite{ThomHudson}. The latter functor, originally introduced by Levine and Morel \cite{LevineMorel}, is a refinement of the other two.
 The articles \cite{GrothendieckBuch} and \cite{ThomHudson} respectively describe the degeneracy loci in $K^0$ and $CK^*$ through the double Grothendieck polynomials $\{\mathfrak{G}_{w}(\mathbf{x},\mathbf{y})\}_{w\in S_\infty}$ of Lascoux--Sch\"utzenberger  \cite{Lascoux1990,HopfLascoux} and the double $\beta$-polynomials $\{\mathfrak{H}^{(\beta)}_{w}(\mathbf{x},\mathbf{y})\}_{w\in S_\infty}$ of Fomin--Kirillov~\cite{GrothendieckFomin}.
   
 Given this state of affairs, it seems natural to wonder whether Kim's result can be interpreted within the framework of degeneracy loci so that it generalises to both $K^0$ and $CK^*$. This is indeed the case.
 
 \begin{theorem}
 Let $Fl_n$ denote the full flag variety of quotient flags of $\Aff_K^n$ and $Q_{n-1}\srarrow\trecd \srarrow Q_1$ be its associated universal flag of quotient bundles. Here $K$ is an algebraic closure of the field $\mathbb{F}_q$. Set $M_i:=Ker(Q_i\srarrow Q_{i-1})$. Then, for every permutation $w\in S_n$ the fundamental class of the closure of the Deligne--Lusztig variety $X(w)$, as an element of $CK^*(Fl_n)$, is given by
 $$\left[\overline{X(w)}\right]_{CK}=\bpol^{(-\beta)}_{ww_0}\Big(q\odot c_1(M_i), c_1(M_{n+1-j}^\vee)\Big)$$
 for $i,j\in\{1,\tred,n\}$. Here $\bpol^{(\beta)}$ stands for the double $\beta$-polynomial of Fomin--Kirillov and the formal multiplication $\odot$ is given by
 $$q\odot x=\sum_{i=1}^q \binom{q}{i}x^i(-\beta)^{i-1},$$
where $\beta\in CK^{-1}(\spec K)$.  
 By respectively setting $\beta$ equal to 0 and 1, one obtains analogous formulas for the Chow ring and for the Grothendieck ring of vector bundles:  
 $$i)\ \left[\overline{X(w)}\right]_{CH}=\spol_{ww_0}\Big(q\cdot c_1(M_i),c_1(M_{n+1-j})\Big) ;\ ii)\  
 \left[\mathcal{O}_{\overline{X(w)}}\right]_{K^0}=\gpol_{ww_0}\Big(1-[M_i^\vee]^q,1-[M_{n+1-j}]\Big).$$
 \end{theorem}   







While \textit{i)} recovers Kim's formula, \textit{ii)} appears to be new.  It is worth stressing that our method has the advantage of highlighting the geometric picture and it does not rely on \cite{HomologyKim}. We expect to be able to apply it to the Deligne--Lusztig varieties of the other classical groups as well. In these cases the formulas should involve Grothendieck analogues of the double Schubert polynomials of Ikeda--Mihalcea--Naruse \cite{IkedaMihalceaNaruse}. 

The paper is structured as follows. Section 2 is devoted to the definition of Deligne--Lusztig varieties, which we then relate to degeneracy loci. In section 3 we provide a quick review of connective $K$-theory and prove two statements that will be needed in section 4 for the proof of the main result. 

\vspace{-0.35cm}
 
\paragraph{Acknowledgements:} This research was conducted in the framework of the research training group
\emph{GRK 2240: Algebro-Geometric Methods in Algebra, Arithmetic and Topology},
which is funded by the DFG. Both authors would like to thank Ulrich Goertz for his helpful comments on an earlier version of this work. 

\vspace{-0.35cm}

\paragraph{Notations and conventions:} 
Throughout this paper $k$ will represent the field $\mathbb{F}_q$, where $q=p^m$ for some prime number $p\in\NN$ and some exponent $m\in \NN$, while $K$ will denote its algebraic closure.
We will denote by $\SM$ the category of smooth schemes over $\spec K$.

\section{Recollections on Deligne--Lusztig varieties and degeneracy loci}

\subsection{Deligne--Lusztig varieties}
Let us begin by recalling the notion of Frobenius endomorphism. For a scheme $X$ defined over $\spec k$ the absolute Frobenius, denoted $F:X\rightarrow X$, is defined in such a way that its associated morphism of topological spaces is just the identity and the map between the structure sheaves raises every section to the $q$-th power. If we consider the base change of $F$ to the algebraic closure, we obtain the relative Frobenius $F_{rel}:\overline{X}\rightarrow \overline{X}$. 
Let us begin with the following elementary lemma.
\begin{lemma}\label{lem qpow}
Let $L$ be a line bundle defined over the $k$-scheme $X$ and denote by $\overline{L}$ and $\overline{X}$ the schemes obtained by base change to $K$. Then one has
$$F_{rel}^* \overline{L}\simeq \overline{L}^{\otimes q}.$$    
\end{lemma}

\begin{proof}
If we consider $\spec K \times_{\spec k}X=:\overline{X}\stackrel{pr_2}\longrightarrow X$, the morphism arising from the base change to $K$, then we have the following identifications.
$$F_{rel}^*\overline{L}=F_{rel}^*(pr_2^* L)= (pr_2\circ F_{rel})^*L=(F\circ pr_2)^*L=pr_2^*(F^*L)\simeq pr_2^*(L^{\otimes q})= (pr_2^* L)^{\otimes q}= \overline{L}^{\otimes q}$$
They follow from the functoriality of pullbacks of bundles and the known fact that pulling back a line bundle along the absolute Frobenius morphism raises it to the $q$-th tensor power.  
\end{proof}

Our interest in $F_{rel}$ is due to the role it plays in the definition of  \textit{Deligne--Lusztig varieties}, a family of locally closed subsets of flag varieties. More precisely, for every positive integer $n$ we consider the variety $Fl_n$, which parametrises the full flags $0\subset U_1 \subset \trecd \subset U_{n-1}\subset\mathbb{A}_{K}^n$ of the $n$-th affine space, where $U_i$ is a vector space of dimension $i$. Please notice that both the affine space and the flag variety will be viewed as schemes over the algebraically closed field $K$.

 In view of the identification between $Fl_n$ and $\mathcal{B}$, the set of all Borel subgroups of $GL_n(K)$, we can subdivide $Fl_n\times Fl_n$ by making use of the Bruhat decomposition of $\mathcal{B}\times \mathcal{B}$. Conjugation by elements of $GL_n(K)$ defines an action on $\mathcal{B}$ and the orbits of the corresponding diagonal action on $\mathcal{B}\times \mathcal{B}$ are indexed by the symmetric group $S_n$, the Weyl group of $GL_n(K)$. In other words, to every $w\in S_n\subseteq GL_n(K)$ we associate $O(w)$, the orbit of $(B, wBw^{-1})$, and one says that two Borel subgroups $B$ and $B'$ are in relative position $w$ whenever $(B,B')\in O(w)$. As no confusion can arise, $O(w)$ will also denote the corresponding orbit inside of $Fl_n\times Fl_n$.  

Let us now consider $\Gamma_{F_{rel}}:Fl_n\rightarrow Fl_n\times Fl_n$, the graph morphism of $F_{rel}$. For every $w\in S_n$, we define the Deligne--Lusztig variety associated to $w$ by setting 
$$X(w):=\Gamma_{F_{rel}}^{-1}(O(w)).$$

\subsection{Degeneracy loci}
We now recall some basic facts concerning degeneracy loci of maps of vector bundles.
Let $X\in \SM$ be a smooth scheme over which is given a morphism $h:E\rightarrow F$ of vector bundles of respective ranks $e$ and $f$. For every integer choice of $0\leq r\leq \min\{e,f\}$, we can construct the degeneracy locus
$$D_r(h):=\left\{x\in X\ |\ \text{rank}\Big(h(x):E(x)\rightarrow F(x)\Big)\leq r\right\}.$$
Its scheme structure is given by regarding it as the zero scheme $Z(\wedge^{r+1}h)$, where $\wedge^{r+1}h$ is interpreted as a section of the bundle $\text{Hom}(\wedge^{r+1}E,\wedge^{r+1}F)$.
We will also consider the following open subset

$$D^\circ_r(h):=\left\{x\in X\ |\ \text{rank}\Big(h(x):E(x)\rightarrow F(x)\Big)= r\right\}=D_r(h)\setminus D_{r-1}(h).$$
 
Both constructions can be generalised to the case of bundles with flags. Assume that $E$ is endowed with a full flag of subbundles $E_\bullet=(E_1\irarrow\trecd E_{e-1}\irarrow E)$ and, similarly, that $F$ comes equipped wih a full flag of quotient bundles $F_\bullet=(F\srarrow F_{f-1}\srarrow \trecd \srarrow F_1)$. Then, to every function $\mathbf{r}:\{1,\tred,f\}\times\{1,\tred, e\}\rightarrow \NN$, we can associate the subscheme
$$\Omega_\mathbf{r}(E_\bullet,F_\bullet, h):= \bigcap_{i,j}D_{\mathbf{r}(j,i)}(h_{i,j}), $$
where $h_{i,j}$ stands for the composition $E_i\irarrow E\rightarrow F \srarrow F_j$. In a similar fashion we can also define $\Omega^\circ_\mathbf{r}(E_\bullet,F_\bullet, h)$.   

We will now consider some important examples, in which we will always set $h=id$. 

\begin{example}
Let $X$ be the flag variety $Fl_n$ associated to the affine space $\mathbb{A}_K^n$ and let $U_\bullet$ be the tautological flag of subbundles of $\underline{\mathbb{A}_K^n}:=Fl_n\times \mathbb{A}_K^n$. Every point $x\in Fl_n$ represents a full flag of vector spaces inside $\mathbb{A}^n_K$ obtained by considering the following restrictions 
$$U_1(x)\irarrow \trecd \irarrow U_{n-1}(x)\irarrow \mathbb{A}_K^n=\underline{\mathbb{A}_K^n}(x).$$
Let us denote this flag by $U_\bullet$, by $Q_\bullet$ the flag of quotient bundles $\underline{\mathbb{A}_K^n}/U_\bullet$ and by $\underline{\mathbb{A}_K^\bullet}$ the flag of trivial subbundles associated to a chosen point $\tilde{x}\in Fl_n$.  
For every permutation $w\in S_n$ one considers the function $\mathbf{r}_w:\{1,\tred, n\}^2\rightarrow \NN$ given by 
$$\mathbf{r}_w(j,i):=\{l\leq j\ |\ w(l)\leq i\}.$$
In this particular setting the degeneracy loci
$\Omega_w:=\Omega_{\mathbf{r}_w}(\underline{\mathbb{A}_K^\bullet},Q_\bullet,id_{Fl_n})$ turn out to be reduced and recover the Schubert varieties, while $\Omega^\circ_w:=\Omega^\circ_{\mathbf{r}_w}(\underline{\mathbb{A}_K^\bullet},Q_\bullet,id_{Fl_n})$ become the Schubert cells (for details see \cite[Lemma 6.1]{FlagsFulton}). It is worth pointing out that with this definition one has $l(w)=\text{codim}_K(\Omega_w,Fl_n)$, where the length function $l$ counts the number of inversions of the permutation $w$. To be more specific, the comparison with the notations used in \cite[Sections 2.2, 2.3]{FultonPragacz} is given by $\Omega_w=X_{ww_0}=Y_{w_0ww_0}.$ 
\end{example}

\begin{example}
The previous example can be generalised as follows. Let $V\rightarrow X$ be a vector bundle of rank $n$ over a smooth base and $V_\bullet$ a full flag of subbundles. As for the flag variety, the associated flag bundle $\pi:\flag V\rightarrow X$ comes equipped with the tautological flag $\mathcal{U}_\bullet$ of subbundles of $\pi^* V$ and with the quotient flag $\mathcal{Q}_\bullet$. The generalisation of Schubert varieties and Schubert cells is then obtained by setting
$\Omega_w:=\Omega_{\mathbf{r}_w}(\pi^* V_\bullet,\mathcal{Q}_\bullet,id_{\flag V})$ and $\Omega^\circ_w:=\Omega^\circ_{\mathbf{r}_w}(\pi^* V_\bullet,\mathcal{Q}_\bullet,id_{\flag V})$.  
\end{example}

\begin{example}\label{ex Orbits}
Let us consider a special case of the previous example. Take $\underline{\mathbb{A}_K^n}\rightarrow Fl_n$ as the given vector bundle $V\rightarrow X$ and $U_\bullet$ as the reference flag $V_\bullet$. In this case $\flag V$ is given by $Fl_n\times Fl_n \stackrel{pr_1}\rightarrow Fl_n$ and it is easy to check that the Bruhat decomposition can be described in terms of Schubert cells. More precisely, one has 
$O(w)=\Omega^\circ_{\mathbf{r}_{ww_0}}(pr_1^* U_\bullet,\mathcal{Q}_\bullet,id_{\flag V})$.
\end{example}

\section{Connective $K$-theory}
The goal of this section is to provide a brief overview of connective $K$-theory and extend to positive characteristic a result of \cite{ThomHudson} which describes the fundamental classes of the Schubert varieties of flag bundles.

Connective $K$-theory, denoted $CK^*:\SM^{op}\rightarrow \mathbf{R}^*$, is a contravariant functor from the category of smooth schemes to graded rings. It refines the Chow ring $CH^*$ and the Grothendieck group of vector bundles $K^0$. Through the years several alternative definitions of $CK^*$ have been proposed. The first, which requires the base field $k$ to satisfy resolution of singularities, is due to Levine--Morel \cite{LevineMorel} who defined it by using algebraic cobordism as the universal oriented cohomology theory with multiplicative formal group law. In \cite{AlgebraicCai} Cai proposed another definition, based on the Gersten complex, which can be used in every characteristic. It is worth mentioning that his theory is actually bigraded, but it contains $CK^*$ as its geometric part. Later, in \cite{ConnectiveDai}, Dai--Levine proposed yet another construction of $CK^*$ for schemes over perfect fields, in the context of motivic homotopy theory. Finally, Anderson modified Cai's definition to build a refined oriented Borel--Moore functor which returns $CK^*$ as its associated operational cohomology theory. This approach was introduced by Anderson in \cite[Appendix A]{K-theoreticAnderson} to describe fundamental classes of degeneracy loci and as a consequence it is the most suited to our needs. 

We will now illustrate the main features of connective $K$-theory. Although $CK^*$ is a contravariant functor it also admits push-forward morphisms $g_*$ for proper maps, exactly as $CH^*$ and $K^0$. These satisfy some expected properties of functorial nature and are compatible with pull-back morphisms $f^*$ through a base change formula whenever $f$ and $g$ are transverse. By combining these two operations one is able to define the first Chern class operator associated to a line bundle $L\rightarrow X$. If $s$ denotes the zero section, then one sets $\widetilde{c_1}(L):=s^*s_*:CK^*(X)\rightarrow CK^{*-1}(X)$ with the first Chern class $c_1(L)$ being the evaluation of this operator on the fundamental class $[X]_{CK}:=1_{CK^*(X)}$. Since $CK^*$ satisfies the projective bundle formula, it is possible to use Grothendieck's method to obtain Chern classes for arbitrary vector bundles. These satisfy the same formal properties of their counterparts in $CH^*$ (\textit{e.g.} the Whitney sum formula and various compatibilities with $f^*$ and $g_*$) with one important exception: it is no longer true that $c_1$ is linear with respect to the tensor product of line bundles. Instead, one has
$$c_1(L\otimes M)=c_1(L)\oplus c_1(M):=c_1(L)+c_1(M)-\beta c_1(L)c_1(M),$$
where $\beta\in CK^{-1}(\spec\,k)$ is identified with the push-forward of the fundamental class of $\Proj^1$ to the point. Actually, as pointed out in \cite[Appendix A.2]{K-theoreticAnderson}, the coefficient ring of $CK^*$ is isomorphic to $\ZZ[\beta]$. Notice that our sign convention for $\beta$ agrees with that of \cite{LevineMorel}, while it is opposite to that of \cite{K-theoreticAnderson}. The class $\beta$ also plays a central role in relating $CK^*$ with $CH^*$ and $K^0$. In fact, setting it equal to 0 allows one to recover the Chow ring, while making it invertible returns the Grothendieck ring. To be more precise one has functorial isomorphisms 
$$CK^*(X)/(\beta)\simeq CH^*(X)\quad \text{ and }\quad (CK^*(X))[\beta^{-1}]\simeq K^0(X)\otimes_\ZZ\ZZ[\beta,\beta^{-1}].$$           

We finish this section with two results that will be used in the main proof. For the first, let us notice that in the language of \cite{LevineMorel}, the operation $\oplus$ should be viewed as the formal group law associated to $CK^*$. Its formal inverse is then given by
$$\ominus x:=-\frac{x}{1-\beta x},$$
so that $(\ominus x) \oplus x=0$. In a similar spirit, one can define a formal multiplication $n\odot x$ by formally adding $n$ times the same element $x$. Since for every line bundle $L$ one has 
\begin{eqnarray*}
c_1(L^{\otimes n})=n\odot c_1(L), 
\end{eqnarray*}
 the following formula will allow us to express the first Chern class of tensor powers of line bundles. 

\begin{lemma}\label{lem multiplication}
For every $n\in \NN$ one has 
$$n\odot x=\sum_{i=1}^n \binom{n}{i}x^i(-\beta)^{i-1}.$$
\end{lemma}

\begin{proof}
The proof is by induction and the statement holds trivially for $n=0,1$. For the induction step we have
\begin{eqnarray*}
(n+1)\odot x&=& x\oplus (n\odot x)=x+(n\odot x)-\beta x (n\odot x)\\
&=& x+\sum_{i=1}^n \binom{n}{i}x^i(-\beta)^{i-1}+\sum_{j=1}^n \binom{n}{j}x^{j+1}(-\beta)^{j}\\
&=& (n+1)x+\sum_{i=2}^n \binom{n}{i}x^i(-\beta)^{i-1}+\sum_{i=2}^n \binom{n}{i-1}x^{i}(-\beta)^{i-1}+x^{n+1}(-\beta)^n\\
&=& (n+1)x+\sum_{i=2}^n \binom{n+1}{i}x^i(-\beta)^{i-1}+x^{n+1}(-\beta)^n=\sum_{i=1}^{n+1} \binom{n+1}{i}x^i(-\beta)^{i-1}.\qedhere
\end{eqnarray*} 
\end{proof}

The following result, which is the positive characteristic counterpart of \cite[Proposition 4.11]{ThomHudson}, expresses the fundamental classes of the Schubert varieties of a full flag bundle in terms of the double $\beta$-polynomials of Fomin--Kirillov \cite{GrothendieckFomin}. These are polynomials in $2n$ variables with coefficients in $\ZZ[\beta]$ which unify the double Schubert and Grothendieck polynomials of Lascoux--Sch\"utzenberger. 

\begin{proposition}\label{prop Schubert}
 Let $V$ be a rank $n$ vector bundle endowed with a full flag of subbundles $V\unddot$ over the smooth scheme $X$. Consider the associated full flag bundle $\pi:\flag V\rightarrow X$ and $Q\unddot$, its universal flag of quotient bundles. For every $w\in S_n$, the fundamental class of the Schubert variety $\Omega_w$ is given by  
$$[\Omega_w]_{CK}=\bpol^{(-\beta)}_{w}\Big(c_1(M_i), c_1\big(\pi^*(L_j^\vee)\big)\Big),$$
where $\bpol^{(-\beta)}_w$ stands for the double $\beta$-polynomial associated to $w$ and we set $L_i:=V_i/V_{i-1}$ and $M_i:=\text{Ker\,}\,(Q_i\rightarrow Q_{i-1})$.
\end{proposition}

\begin{proof}
Although it requires some preliminary verifications, the proof is essentially an adaptation of the one for $CH^*$ given in \cite{FlagsFulton} by Fulton and later generalised to any oriented cohomology theory in \cite{ThomHudson}. First one verifies that, as a ring, $CK^*(\flag V)$ is isomorphic to $CK^*(X)[x_1,\tred,x_n]$ modulo the ideal generated by the elements $e_i(x)-c_i(V)$ with $i\in\{1,\tred,n\}$ and $e_i(x)$ being the $i$-th elementary symmetric function. As pointed out in \cite[Theorem 2.6]{SchubertHornbostel}, such isomorphism holds as long as $CK^*$ satisfies the projective bundle formula, which it does (see \cite[Theorem 6.3]{AlgebraicCai}).

The second step consists in verifying that the push-pull operators $\pi_i^*\pi_{i*}$ on $CK^*(\flag V)$ coincide with $\phi_i$, the $\beta$-divided difference operators of Fomin--Kirillov. Here $\pi_i:\flag V \rightarrow \flag_{\hat i} V$ is the projection onto the partial flag bundle in which the $i$-th flag has been forgotten. Since $\flag V\stackrel{\pi_i}\longrightarrow \flag_{\hat i} V$ can be viewed as the projective bundle of a vector bundle of rank 2, $\pi_{i*}$ is completely determined by the images of 1 and $x_i$. By making use of \cite[Appendix A.1, \S\ Chern classes (c)]{K-theoreticAnderson} one can easily check that on these elements the two operators actually coincide.  

These preliminary facts being checked, we can move on to the actual proof, which is by induction on the length of $w_0w$. First one verifies the formula for the longest element, whose associated Schubert variety is isomorphic to the base scheme $X$. Since $\Omega_{\omega_0}$ can be described as the zero scheme of a regular section of a bundle, in view of \cite[Appendix A.2, \S\ Chern classes]{K-theoreticAnderson} the fundamental class $[\Omega_{w_0}]_{CK}$ is given by the top Chern class of the bundle in question. More precisely, one has
$$[\Omega_{w_0}]_{CK}=\prod_{i+j\geq n} c_1(M_i\otimes \pi^*(L_j^\vee))=\prod_{i+j\geq n} c_1(M_i)\oplus c_1(\pi^*(L_j^\vee))=\bpol^{(-\beta)}_{w_0}\Big(c_1(M_i), c_1\big(\pi^*(L_j^\vee)\big)\Big).$$

Finally, for the inductive step one considers a minimal decomposition of $w_0w$ into elementary transpositions $s_{i_1}s_{i_2}\trecd s_{i_l}$ to which we associate the $l$-tuple $I=(i_1,\dots,i_l)$. Recall that every such tuple gives rise to $R_I$, a desingularisation of $\Omega_w$ known as Bott--Samelson resolution. In view of the recursive construction of $R_I\stackrel{\varphi_I}\longrightarrow \flag V$ (see \cite[Appendix C]{FultonPragacz} for details) and of the known compatibilities of pushforward and pullback maps for trasverse morphisms, we have  
\begin{eqnarray*}
\varphi_{I*}[R_I]_{CK}&=&\pi_{i_l}^*\pi_{i_l*}\cdots \pi_{i_1}^*\pi_{i_1*}[\Omega_{w_0}]_{CK}=\phi_{i_l} \cdots \phi_{i_1}\bpol^{(-\beta)}_{w_0}\Big(c_1(M_i), c_1\big(\pi^*(L_j^\vee)\big)\Big)\\
&=&\bpol^{(-\beta)}_{w}\Big(c_1(M_i), c_1\big(\pi^*(L_j^\vee)\big)\Big),
\end{eqnarray*}
where the last step follows from the inductive definition of double $\beta$-polynomials. To finish the proof one observes that, since Schubert varieties have rational singularities, the left hand side of the preceeding equation actually coincides with $[\Omega_w]_{CK}$ (see \cite[Remark 1.2]{K-theoreticAnderson}).
\end{proof}

\section{Main result}

\begin{theorem}\label{thm DL}
Let $Fl_n$ be the variety of full flags contained in $\Aff_K^n$ and $Q\unddot=(Q_{n-1},\dots,Q_1)$ its universal flag of quotient bundles. For $i\in \{1,\dots,n\}$ set $M_i:=\text{Ker}(Q_i\srarrow Q_{i-1})$, where $Q_n=\Aff_K^n$ and $Q_0=0$. 
Then, for every $w\in S_n$ we have that, as an element of $CK^*(Fl_n)$, the fundamental class of the closure of the Deligne-Lusztig variety $X(w)$ is given by 
$$\left[\overline{X(w)}\right]_{CK}=\bpol^{(-\beta)}_{ww_0}\left(q\odot c_1(M_i),\ominus c_1(M_{n+1-j})\right)$$
for $i,j\in\{1,\tred,n\}$. Here $\bpol^{(-\beta)}_{ww_0}$ stands for the double $\beta$-polynomial associated to $ww_0$, while 
$$q\odot c_1(M_i)=\sum_{j=1}^q \binom{q}{j} c_1(M_i)^j (-\beta)^{j-1} 
\text{\  and\  }\ominus c_1(M_{n+1-j})=-\frac{c_1(M_{n+1-j})}{1-\beta c_1(M_{n+1-j})}.$$
\end{theorem}
\begin{proof}
Let us begin by recalling that the $\overline{X(w)}$ coincides with the union of all the $X(v)$ for $v\leq w$ in the Bruhat order and that the same holds for $\overline{O(w)}$ as well. As a consequence one has that $\Gamma_{F_{rel}}^{-1}(\overline{O(w)})=\overline{X(w)}$ and, since both varieties have the same codimension in the respective ambient spaces, the fundamental class of $\overline{X(w)}$ can be computed as the pullback $\Gamma^*_{F_{rel}}[\overline{O(w)}]$.
Now we want to interpret $\Gamma_{F_{rel}}:Fl_n\rightarrow Fl_n\times Fl_n$ using the universal property of $Fl_n\times Fl_n$ viewed as the flag bundle $\flag \underline{\mathbb{A}_K^n}$, as explained in Example \ref{ex Orbits}. With our conventions it parametrises the full flags of quotient bundles of the trivial bundle $\underline{\mathbb{A}_K^n}$. It is easy to see that $\Gamma_{F_{rel}}$ corresponds precisely to the full flag $F_{rel}^*Q_\bullet$ or, in other words, that with the notations of Example \ref{ex Orbits} one has 
$\Gamma^*_{F_{rel}}\mathcal{Q_\bullet}=F^*_{rel} Q_\bullet$. To summarise, we have the following chain of equalities  
\begin{eqnarray*}
\left[\overline{X(w)}\right]_{CK}&=&\Gamma^*_{F_{rel}}\left[\overline{O(w)}\right]_{CK}=\Gamma_{F_{rel}}^*[\Omega_{ww_0}(\underline{\Aff}_K^n)]_{CK}
=\Gamma_{F_{rel}}^*\bpol^{(-\beta)}_{ww_0}\Big(c_1(M_i'), c_1\big(pr_1^*(U_j/U_{j-1})^\vee\big)\Big)\\
&=&\bpol^{(-\beta)}_{ww_0}\Big(c_1(\Gamma_{F_{rel}}^*M_i')\big), c_1\big(\Gamma_{F_{rel}}^*pr_1^*(U_j/U_{j-1})^\vee\big)\Big)
=\bpol^{(-\beta)}_{ww_0}\Big(c_1\big(F_{rel}^*M_i\big), c_1\big((U_j/U_{j-1})^\vee\big)\Big),
\end{eqnarray*}
where the second step uses Proposition \ref{prop Schubert} and each $M'_i$ is the line bundle $\textit{Ker}(\mathcal{Q}_i\srarrow \mathcal{Q}_{i-1})$ arising from the universal flag $\mathcal Q_\bullet$ over $Fl_n\times Fl_n$. To finish the proof it now suffices to make use of Lemma \ref{lem qpow} and Lemma \ref{lem multiplication}.
\begin{eqnarray}\label{eq last}
\left[\overline{X(w)}\right]_{CK}=\bpol^{(-\beta)}_{ww_0}\Big(c_1(M_i^{\otimes q}), c_1(M_{n+1-j}^\vee)\Big)
=\bpol^{(-\beta)}_{ww_0}\Big(q\odot c_1(M_i),\ominus c_1(M_{n+1-j})\Big). \qedhere
\end{eqnarray}

\end{proof}

By specialising our formula to the Chow ring and to the Grothendieck ring, we obtain the following corollary, the first formula of which recovers the first case of \cite[Proposition 6.2]{HomologyKim}.

\vspace{-0.1 cm}

\begin{corollary}\label{cor main}
Under the hypothesis of Theorem \ref{thm DL}, we have the following formulas, respectively describing the Chow ring fundamental class $\left[\overline{X(w)}\right]_{CH}$ and the class of the structure sheaf $\mathcal{O}_{\overline{X(w)}}$:
$$i)\ \left[\overline{X(w)}\right]_{CH}=\spol_{ww_0}\Big(q\cdot c_1(M_i),c_1(M_{n+1-j})\Big) ; \ ii)\ 
\left[\mathcal{O}_{\overline{X(w)}}\right]_{K^0}=\gpol_{ww_0}\Big(1-[M_i^\vee]^q,1-[M_{n+1-j}]\Big).$$
\end{corollary}

\begin{proof}
The statement for $CH^*$ follows directly from that of $CK^*$ by setting $\beta=0$ and recalling that $\bpol^{(0)}_{w}(\mathbf{x},\mathbf{y})=\mathfrak{S}_w(\mathbf{x},-\mathbf{y})$. To obtain the second formula we consider the middle equation of (\ref{eq last}) and set $\beta=1$. Since $\bpol^{(-1)}_{w}(\mathbf{x},\mathbf{y})=\mathfrak{G}_w(\mathbf{x},\mathbf{y})$, one gets
$$\left[\mathcal{O}_{\overline{X(w)}}\right]_{K^0}=\gpol_{ww_0}\Big(c_1(M_i^{\otimes q}), c_1(M_{n+1-j}^\vee)\Big)$$
and the statement then follows from the well known fact that in the Grothendieck ring of vector bundles one has $c_1(L)=1-[L^\vee]$ for all line bundles. 
\end{proof}
\phantom{a}
\vspace{-0.9 cm}

\begin{remark}
The sign mismatch between the corresponding formulas of Corollary \ref{cor main} and \cite[Proposition 6.2]{HomologyKim} is due to different conventions for the generators of $CH^*(Fl_n)$ which ultimately arise from dual constructions of $Fl_n$. The two formulas coincide provided one makes the change of variables $x_i \mapsto -x_{n+1-i}$.  
\end{remark}

\vspace{-1.1 cm}
\phantom{a}
\bibliographystyle{acm}
\bibliography{references}

\def\cprime{$'$}
\begin{thebibliography}{10}

\bibitem{K-theoreticAnderson}
{\sc Anderson, D.}
\newblock {$K$}-theoretic {C}hern class formulas for vexillary degeneracy loci.
\newblock {\em Adv. Math. 350\/} (2019), 440--485.

\bibitem{GrothendieckBuch}
{\sc Buch, A.~S.}
\newblock Grothendieck classes of quiver varieties.
\newblock {\em Duke Math. J. 115}, 1 (2002), 75--103.

\bibitem{AlgebraicCai}
{\sc Cai, S.}
\newblock Algebraic connective {$K$}-theory and the niveau filtration.
\newblock {\em J. Pure Appl. Algebra 212}, 7 (2008), 1695--1715.

\bibitem{ConnectiveDai}
{\sc Dai, S., and Levine, M.}
\newblock Connective algebraic {$K$}-theory.
\newblock {\em J. K-Theory 13}, 1 (2014), 9--56.

\bibitem{RepresentationDeligne}
{\sc Deligne, P., and Lusztig, G.}
\newblock Representations of reductive groups over finite fields.
\newblock {\em Ann. of Math. (2) 103}, 1 (1976), 103--161.

\bibitem{GrothendieckFomin}
{\sc Fomin, S., and Kirillov, A.~N.}
\newblock Grothendieck polynomials and the {Y}ang-{B}axter equation.
\newblock In {\em Formal power series and algebraic combinatorics/{S}\'eries
  formelles et combinatoire alg\'ebrique}. DIMACS, Piscataway, NJ, 1994,
  pp.~183--189.

\bibitem{FlagsFulton}
{\sc Fulton, W.}
\newblock Flags, {S}chubert polynomials, degeneracy loci, and determinantal
  formulas.
\newblock {\em Duke Math. J. 65}, 3 (1992), 381--420.

\bibitem{FultonPragacz}
{\sc Fulton, W., and Pragacz, P.}
\newblock {\em Schubert varieties and degeneracy loci}, vol.~1689 of {\em
  Lecture Notes in Mathematics}.
\newblock Springer-Verlag, Berlin, 1998.
\newblock Appendix J by the authors in collaboration with I. Ciocan-Fontanine.

\bibitem{SchubertHornbostel}
{\sc Hornbostel, J., and Kiritchenko, V.}
\newblock Schubert calculus for algebraic cobordism.
\newblock {\em J. Reine Angew. Math. 656\/} (2011), 59--85.

\bibitem{ThomHudson}
{\sc Hudson, T.}
\newblock A {T}hom-{P}orteous formula for connective {$K$}-theory using
  algebraic cobordism.
\newblock {\em J. K-Theory 14}, 2 (2014), 343--369.

\bibitem{IkedaMihalceaNaruse}
{\sc Ikeda, T., Mihalcea, L.~C., and Naruse, H.}
\newblock Double {S}chubert polynomials for the classical groups.
\newblock {\em Adv. Math. 226}, 1 (2011), 840--886.

\bibitem{HomologyKim}
{\sc Kim, D.}
\newblock {Homology Class of a Deligne-Lusztig Variety and Its Analogs}.
\newblock {\em Int. Math. Res. Not.\/} (04 2018).

\bibitem{ClassesLascoux}
{\sc Lascoux, A.}
\newblock Classes de {C}hern des vari\'et\'es de drapeaux.
\newblock {\em C. R. Acad. Sci. Paris S\'er. I Math. 295}, 5 (1982), 393--398.

\bibitem{Lascoux1990}
{\sc Lascoux, A.}
\newblock Anneau de {G}rothendieck de la vari\'et\'e de drapeaux.
\newblock In {\em The {G}rothendieck {F}estschrift, {V}ol.\ {III}}, vol.~88 of
  {\em Progr. Math.} Birkh\"auser Boston, Boston, MA, 1990, pp.~1--34.

\bibitem{LascouxSchutzenberger}
{\sc Lascoux, A., and Sch{\"u}tzenberger, M.-P.}
\newblock Polyn\^omes de {S}chubert.
\newblock {\em C. R. Acad. Sci. Paris S\'er. I Math. 294}, 13 (1982), 447--450.

\bibitem{HopfLascoux}
{\sc Lascoux, A., and Sch{\"u}tzenberger, M.-P.}
\newblock Structure de {H}opf de l'anneau de cohomologie et de l'anneau de
  {G}rothendieck d'une vari\'et\'e de drapeaux.
\newblock {\em C. R. Acad. Sci. Paris S\'er. I Math. 295}, 11 (1982), 629--633.

\bibitem{LevineMorel}
{\sc Levine, M., and Morel, F.}
\newblock {\em Algebraic cobordism}.
\newblock Springer Monographs in Mathematics. Springer, Berlin, 2007.

\bibitem{CharactersLusztig}
{\sc Lusztig, G.}
\newblock {\em Characters of reductive groups over a finite field}, vol.~107 of
  {\em Annals of Mathematics Studies}.
\newblock Princeton University Press, Princeton, NJ, 1984.

\end{thebibliography}

\begin{small}
{\scshape
\noindent Thomas Hudson, Fachgruppe Mathematik
und Informatik, Bergische Universit\"{a}t Wuppertal, Gau{\ss}strasse 20, 42119 Wuppertal, Germany
}
\end{small}

{\textit{email address}: \tt{hudson@math.uni-wuppertal.de}}

\

\begin{small}
{\scshape
\noindent Dennis Peters, Fachgruppe Mathematik
und Informatik, Bergische Universit\"{a}t Wuppertal, Gau{\ss}strasse 20, 42119 Wuppertal, Germany
}
\end{small}

{\textit{email address}: \tt{dpeters@math.uni-wuppertal.de}}
\end{document}